\documentclass[reqno]{amsart}
\usepackage{graphicx} 
\usepackage{xcolor}
\usepackage[a4paper, total={6in, 8in}]{geometry}
\usepackage{amsmath}
\numberwithin{equation}{section}
\usepackage{amssymb}
\usepackage{amsthm}
\usepackage{enumerate}   
\theoremstyle{plain}
\newtheorem{theorem}{Theorem}[section]

\newtheorem{remark}{Remark}[section]
\newtheorem{proposition}{Proposition}[section]
\newtheorem{corollary}{Corollary}[section]
\newcommand{\sign}{\text{ sign}}
\title{Grand Net Spaces and Applications to Integral Operators}
\author[D. Suragan]{Durvudkhan Suragan} 
\address{$^1$Department of Mathematics, Nazarbayev University 53 Kabanbay Batyr Ave, Astana 010000 Kazakhstan.}
\email{durvudkhan.suragan@nu.edu.kz}
\author[M. A. Zaighum]{Muhammad Asad Zaighum}
\address{$^2$Department of Mathematics
Nazarbayev University
53 Kabanbay Batyr Ave, Astana 010000
Kazakhstan.\\
Department of Mathematics and Statistics,
Riphah International University, Islamabad, Pakistan.}
\email{zaighum.asad@nu.edu.kz; asad.zaighum@riphah.edu.pk}

\begin{document}

\begin{abstract}
This paper introduces the concept of grand net spaces, a new framework that provides a unified setting for studying various function spaces. Building on the seminal works of \cite{iwaniec1992integrability} and \cite{nursultanov1998net}, we define grand net spaces and establish their key properties, including embedding results, norm equivalences, and interpolation theorems. We prove that these spaces coincide with grand Lorentz spaces under certain conditions and derive boundedness criteria for integral operators acting on grand net spaces. The latter extends the Nursultanov-Tikhonov theorem established in \cite{nursultanov2011net}. 

\end{abstract}

\maketitle

\section{Introduction}
During the past three decades, many researchers have extensively studied grand Lebesgue spaces $L^{p)}(\Omega)$ and their generalizations due to their wide-ranging applications in problems related to PDEs.  These spaces extend the scale of classical Lebesgue spaces. The grand Lebesgue spaces were introduced in 1992 by Iwaniec and Sbordone \cite{iwaniec1992integrability} in connection with the study of integrability properties of Jacobian.  Later, Greco et al. \cite{greco} considered a more general variant of these spaces $L^{p),\theta}(\Omega)$ in the study of existence and uniqueness of the solution of $p$-harmonic equation
$$div|\nabla v|^{p-2}\nabla v=div f.$$

The first systematic study of the structural properties of grand Lebesgue spaces was carried out by Fiorenza et al. in \cite {fiorenza2000grand} and \cite{fiorenza2004grand}. In the context of variable exponent Lebesgue spaces, Kokilashvili et al. \cite{KOK-GV} introduced these spaces and established the boundedness of maximal and Calder\'{o}n Zygmund operators. Later, a more general version of these spaces was studied by Anatriello et al. in
\cite{AN-F}. A comprehensive study of the boundedness of integral operators in non--standard function spaces, including grand--type spaces, is presented in monographs \cite{kokilashviliintegral}, \cite{kokilashviliintegral1}, and \cite{kokilashviliintegral2}. For other recent works in these spaces, we also refer to \cite{fiorenza2018grand}, \cite{meskhi2024predual}, \cite{rafeiro2023grand}, \cite{rafeiro2023local}, \cite{rafeiro2022grand}, \cite {restrepo2021hardy} and references cited therein. Recently, Nursultanov et al. have studied the grand Lorentz space in \cite{DHN} which extends some results in Lorentz spaces.

Let $\Omega\subset{\mathbb{R}^n}$ with $|\Omega|=1$. Let $M^*$ denote the collection of subsets of $\mathbb{R}^n$ with finite positive measure,
\begin{equation*}
M^*=\{\omega\subset\Omega: 0<|\omega|<\infty\}.
\end{equation*}
We refer to $M$ as a net, which is a fixed subset of $M^*$. For a function $f$ integrable on each element of the net $M$, the average function is defined as:
\begin{equation*}
    \bar{f}(t,M)=\sup\limits_{\substack{\omega\in M\\ |\omega|>t}}\frac{1}{|\omega|}\left|\int\limits_{\omega}f(x)dx\right|,
\end{equation*}
with usual consideration that if $\sup\{|\omega|:\omega\in M\}=\alpha<\infty$ and $t>\alpha$, then we set $\bar{f}(t,M)=0$.
Net spaces were introduced in 1998 by Nursultanov \cite{nursultanov1998net}. Formally, given $0<p,q\le \infty$, the net space $N_{p,q}(M)$ is the collection of functions $f$ such that:
\begin{equation*}
\|f\|_{N_{p,q}(M)}=\left(\int\limits_{0}^{\infty}\left(t^{1/p}\bar{f}(t,M)\right)^{q}\frac{dt}{t}\right)^{1/q}<\infty,
\end{equation*}
if $q<\infty$ and
\begin{equation*}
\|f\|_{N_{p,\infty}(M)}=\sup\limits_{t>0}t^{1/p}\bar{f}(t,M)<\infty,
\end{equation*} 
if $q=\infty$.

In \cite{nursultanov1998net}, the author studied the structural properties of these spaces. Reverse Hardy--Littlewood--type inequalities were also proved in the same work. Later in \cite{nursultanov2011net}, the authors considered a general version of the net spaces by considering the function space defined in the framework of general measures. A general criterion for the boundedness of integral operators was also proved on net spaces. In the same work, integral operators on weighted Lebesgue and Lorentz spaces were also discussed. The interpolation of net spaces was discussed in \cite{nursultanov1998net} and \cite{nursultanov2011net}. For interpolation result on anistropic net space we refer to \cite{bashirova2024interpolation}.

This paper is divided into four sections. In Section 2, we introduce the notion of grand net spaces. Some embedding results, its coincidence with grand Lorentz spaces under special assumption on nets is established. We also derive H\"{o}lder type inequality.  In Section 3, we establish real interpolation results for these spaces. In Section 4, we prove a criterion for the boundedness of integral operators in grand net spaces.

Throughout this paper, constants often different within the same series of inequalities, will be denoted by $c$ or $C$; by the symbol $p'$ we denote the function $\frac{p}{p-1}$, $1<p<\infty$; the relation $a\asymp b$ indicates that there exist positive constants $c_1$ and $c_2$ such that $c_1 a\leq b \leq c_2 a$.\\

\section{Grand Net Spaces}
Let $\theta\in\mathbb{R}$ and $\Omega\subset{\mathbb{R}}^n$ with $|\Omega|=1$. Let $M$ be a net of $\Omega$, then the grand net spaces $GN^{\theta}_{p,q}(M)$ is defined by:
\begin{equation}
   \left\|f\right\|_{{GN}^{\theta}_{p,q}(M)} = 
\begin{cases} 
\sup\limits_{0<\varepsilon\le1}\varepsilon^{\theta}\left(\int\limits_{0}^{1}\left(t^{\frac{1}{p}+\varepsilon}\bar{f}(t,M)\right)^{q}\frac{dt}{t}\right)^{\frac{1}{q}} & \text{for } \theta\ge0 \text{ and } 0<p<\infty,  \\ 
\inf\limits_{0<\varepsilon\le\frac{1}{p}}\varepsilon^{\theta}\left(\int\limits_{0}^{1}\left(t^{\frac{1}{p}-\varepsilon}\bar{f}(t,M)\right)^{q}\frac{dt}{t}\right)^{\frac{1}{q}} & \text{for } \theta<0\text{ and }0<p<\infty,  \\ 
\sup\limits_{0<\varepsilon\le1}\varepsilon^{\theta}\left(\int\limits_{0}^{1}\left(t^{\varepsilon}\bar{f}(t,M)\right)^{q}\frac{dt}{t}\right)^{\frac{1}{q}} & \text{for } \theta\ge0\text{ and } p=\infty, \\ 
\end{cases} 
\end{equation}
if $0<q<\infty$, and
\begin{equation}
   \left\|f\right\|_{GN^{\theta}_{p,\infty}(M)} = 
\begin{cases} 
\sup\limits_{0<\varepsilon\le1}\sup\limits_{0<t<1}\varepsilon^{\theta}t^{\frac{1}{p}+\varepsilon}\bar{f}(t,M) & \text{for }  \theta\ge0\text{ and }0<p<\infty,  \\ 
\inf\limits_{0<\varepsilon\le\frac{1}{p}}\sup\limits_{0<t<1}\varepsilon^{\theta}t^{\frac{1}{p}-\varepsilon}\bar{f}(t,M) & \text{for }  \theta<0\text{ and }0<p<\infty,   \\ 
\sup\limits_{0<\varepsilon\le1}\sup\limits_{0<t<1}\varepsilon^{\theta}t^{\varepsilon}\bar{f}(t,M) & \text{for } \theta\ge0\text{ and }p=\infty,  \\
\end{cases} 
\end{equation}
if $q=\infty$.

It is noted that when $\theta=0$ these spaces coincide with classical net spaces, i.e. $GN^{0}_{p,q}(M)=N_{p,q}(M)$. Hence, this space can be considered as a generalization of net spaces. Moreover for $\theta>0$, the following nested embeddings can also be verified using the definition of grand net spaces,
\begin{equation}
GN^{-\theta}_{p,q}(M) \hookrightarrow N_{p,q}(M)\hookrightarrow GN^{\theta}_{p,q}(M),
\end{equation}
where $0<p,q\le\infty$. 
The following theorem gives some simple formulae to compute the norm of functions in grand net spaces.
\begin{theorem}\label{t1}
Let $|\Omega|=1$, $s\in(1,\infty)$ and $0<p<\infty$,  then for $\theta>0$ we have
\begin{align}
    \|f\|_{GN^{\theta}_{p,\infty}}&\asymp\sup_{t>0}\frac{t^{\frac{1}{p}}}{|\ln{t}|^{\theta}}\bar{f}(t,M),\\
      \|f\|_{GN^{-\theta}_{p,\infty}}&\asymp\sup_{t>0}{t^{\frac{1}{p}}}{|\ln{t}|^{\theta}}\bar{f}(t,M),\\
       \|f\|_{GN^{\theta}_{p,s}}&\lesssim\left(\int_{0}^{1}\left(\frac{t^{\frac{1}{p}}}{|\ln{t}|^{\theta}}\bar{f}(t,M)\right)^{s}\frac{dt}{t}\right)^{\frac{1}{s}},\label{L6}\\
       \|f\|_{GN^{-\theta}_{p,s}}&\gtrsim\left(\int_{0}^{1}\left(t^{\frac{1}{p}}|\ln{t}|^{\theta}\bar{f}(t,M)\right)^{s}\frac{dt}{t}\right)^{\frac{1}{s}}.
\end{align}
\end{theorem}
\begin{proof}
    Consider the function $u(\varepsilon)=\varepsilon^{\theta}t^{\frac{1}{p}+\varepsilon \sign{\theta}}$ for $0<t\le1$ and $0<\varepsilon\le1$. Solving equation ${u'}(\varepsilon)=0$ gives $\varepsilon=\frac{|\theta|}{|\ln t|}$. Hence, we have
    \begin{equation*}
\sup\limits_{0<\varepsilon\le1}u(\varepsilon)\asymp\frac{t^{\frac{1}{p}}}{|\ln t|^{\theta}},
    \end{equation*}
   for $\theta >0,$ and 
     \begin{equation*}
    \inf\limits_{0<\varepsilon\le\frac{1}{p}}u(\varepsilon)\asymp{t^{\frac{1}{p}}}{|\ln t|^{\theta}}
    \end{equation*}
    for $\theta<0$.

    Therefore,
    \begin{align*}
     \|f\|_{GN^{\theta}_{p,\infty}}&=\sup_{0<\epsilon\le 1}\varepsilon^{\theta}\sup_{t>0}t^{\frac{1}{p}+\varepsilon}\bar{f}(t,M)\\
     &=\sup_{t>0}\left(\sup_{0<\epsilon\le 1}\varepsilon^{\theta} t^{\frac{1}{p}+\varepsilon}\right)\bar{f}(t,M)\\
     &\asymp\sup_{t>0}\frac{t^{\frac{1}{p}}}{|\ln{t}|^{\theta}}\bar{f}(t,M).\\
    \end{align*}
    Similarly,
    \begin{equation*}
     \|f\|_{GN^{-\theta}_{p,\infty}}\asymp\sup_{t>0}{t^{\frac{1}{p}}}{|\ln{t}|^{\theta}}\bar{f}(t,M).
     \end{equation*}
    To show \eqref{L6} we have
    \begin{align*}
\|f\|_{GN^{\theta}_{p,s}}&=\sup_{0<\varepsilon\le1}\varepsilon^{\theta}\left(\int_{0}^{1}\left(t^{\frac{1}{p}+\varepsilon}\bar{f}(t,M)\right)^{s}\frac{dt}{t}\right)^{\frac{1}{s}}\\
     &\le\left(\int\limits_{0}^{1}\left(\sup\limits_{0<\varepsilon\le1}\varepsilon^{\theta}t^{\frac{1}{p}+\varepsilon}\bar{f}(t,M)\right)^s\frac{dt}{t}\right)^{\frac{1}{s}}\\
     &\asymp\left(\int_{0}^{1}\left(\frac{t^{\frac{1}{p}}}{|\ln{t}|^{\theta}}\bar{f}(t,M)\right)^{s}\frac{dt}{t}\right)^{\frac{1}{s}}.
    \end{align*}
   Additionally, we get 
    \begin{align*}
     \|f\|_{GN^{-\theta}_{p,s}}&=\inf_{0<\varepsilon\le\frac{1}{p}}\varepsilon^{-\theta}\left(\int_{0}^{1}\left(t^{\frac{1}{p}-\varepsilon}\bar{f}(t,M)\right)^{s}\frac{dt}{t}\right)^{\frac{1}{s}}\\
     &\ge\left(\int\limits_{0}^{1}\left(\inf_{0<\varepsilon\le\frac{1}{p}}\varepsilon^{-\theta}t^{\frac{1}{p}-\varepsilon}\bar{f}(t,M)\right)^s\frac{dt}{t}\right)^{\frac{1}{s}}\\
    &\asymp\left(\int_{0}^{1}\left({t^{\frac{1}{p}}}{|\ln{t}|^{\theta}}\bar{f}(t,M)\right)^{s}\frac{dt}{t}\right)^{\frac{1}{s}}.
    \end{align*}
\end{proof}
The following proposition gives the embeddings of the grand net spaces with respect to exponents and the nets.
\begin{proposition}\label{P1}
\begin{enumerate}[(i)]
\item If $\theta\in\mathbb{R}$, $p,s\in(0,\infty]$ and $M_1\subset M_2$ then $GN^{\theta}_{p,s}(M_2)\hookrightarrow GN^{\theta}_{p,s}(M_1)$.
\item If $\theta\in\mathbb{R}$ such that $\theta\le \theta_{1}$ then $GN^{\theta}_{p,s}(M)\hookrightarrow GN^{\theta_1}_{p,s}(M)$.
\item If $\theta\in\mathbb{R}$, $1\le s<s_1$ then $GN^{\theta}_{p,s}(M)\hookrightarrow GN^{\theta}_{p,s_1}(M)$.
\item If $p<p_1$ and $s<s_1$ then $GN^{\theta}_{p_1,s_1}(M)\hookrightarrow GN^{\theta}_{p,s}(M)$.
\item For $0<\delta<1$ and $\theta>0$. We have
\begin{equation}\label {A1}
\left\|f\right\|_{GN^{\theta}_{p,q}(M)} \asymp\sup\limits_{0<\varepsilon\le\delta}\varepsilon^{\theta}\left(\int\limits_{0}^{1}\left(t^{\frac{1}{p}+\varepsilon}\bar{f}(t,M)\right)^{q}\frac{dt}{t}\right)^{\frac{1}{q}},
\end{equation}
and for $0<\delta<\frac{1}{p}$, 
\begin{equation}\label {A2}
\left\|f\right\|_{GN^{-\theta}_{p,q}(M)} \asymp\inf\limits_{0<\varepsilon\le\delta}\varepsilon^{-\theta}\left(\int\limits_{0}^{1}\left(t^{\frac{1}{p}-\varepsilon}\bar{f}(t,M)\right)^{q}\frac{dt}{t}\right)^{\frac{1}{q}}.
\end{equation}
\end{enumerate}    
\end{proposition}
\begin{proof}
The proof of $(i)$ follows immediately from the fact that if $M_1\subset M_2$, then $\bar{f}(t,M_1)\le\bar{f}(t,M_2)$.
 Embeddings in $(ii)-(iv)$ follow from the definition of grand net spaces and its properties established in \cite{nursultanov1998net}.  
To prove \eqref{A1} we follow the arguments given mainly in \cite{fiorenza2004grand}. Note that for $0<\delta<1$  we choose $\eta>1$ such that $\eta\delta=1$:
\begin{align*}
\sup\limits_{0<\varepsilon\le\delta}\varepsilon^{\theta}\left(\int\limits_{0}^{1}\left(t^{\frac{1}{p}+\varepsilon}\bar{f}(t,M)\right)^{q}\frac{dt}{t}\right)^{\frac{1}{q}}&=\sup\limits_{0<\frac{\sigma}{\eta}\le\delta}\left(\frac{\sigma}{\eta}\right)^{\theta}\left(\int\limits_{0}^{1}\left(t^{\frac{1}{p}+\frac{\sigma}{\eta}}\bar{f}(t,M)\right)^{q}\frac{dt}{t}\right)^{\frac{1}{q}}\\
&=\eta^{-\theta}\sup\limits_{0<\sigma\le\eta\delta}\sigma^{\theta}\left(\int\limits_{0}^{1}\left(t^{\frac{1}{p}+\frac{\sigma}{\eta}}\bar{f}(t,M)\right)^{q}\frac{dt}{t}\right)^{\frac{1}{q}}\\
&\ge \eta^{-\theta}\sup\limits_{0<\sigma\le 1}\sigma^{\theta}\left(\int\limits_{0}^{1}\left(t^{\frac{1}{p}+\sigma}\bar{f}(t,M)\right)^{q}\frac{dt}{t}\right)^{\frac{1}{q}}\\
&=c \left\|f\right\|_{GN^{\theta}_{p,q}(M)}. 
\end{align*}
 For $0<\delta<\frac{1}{p}$  we choose $\eta>1$ such that $\eta\delta=\frac{1}{p}$
\begin{align*}
\inf\limits_{0<\varepsilon\le\delta}\varepsilon^{-\theta}\left(\int\limits_{0}^{1}\left(t^{\frac{1}{p}-\varepsilon}\bar{f}(t,M)\right)^{q}\frac{dt}{t}\right)^{\frac{1}{q}}&=\inf\limits_{0<\frac{\sigma}{\eta}\le\delta}\left(\frac{\sigma}{\eta}\right)^{-\theta}\left(\int\limits_{0}^{1}\left(t^{\frac{1}{p}-\frac{\sigma}{\eta}}\bar{f}(t,M)\right)^{q}\frac{dt}{t}\right)^{\frac{1}{q}}\\
&=\eta^{\theta}\inf\limits_{0<\sigma\le\eta\delta}\sigma^{-\theta}\left(\int\limits_{0}^{1}\left(t^{\frac{1}{p}-\frac{\sigma}{\eta}}\bar{f}(t,M)\right)^{q}\frac{dt}{t}\right)^{\frac{1}{q}}\\
&\le \eta^{\theta}\inf\limits_{0<\sigma\le \frac{1}{p}}\sigma^{-\theta}\left(\int\limits_{0}^{1}\left(t^{\frac{1}{p}-\sigma}\bar{f}(t,M)\right)^{q}\frac{dt}{t}\right)^{\frac{1}{q}}\\
&=c \left\|f\right\|_{GN^{-\theta}_{p,q}(M)}. 
\end{align*}
\end{proof}

\begin{proposition}
    Let $0<\theta<\theta_1$ and $q<q_1<\infty$ such that $\theta_1-\frac{1}{q}=\theta-\frac{1}{q_1}$. Then $GN^{\theta}_{p,q_1}(M) \hookrightarrow GN^{\theta_1}_{p,q}(M)$. 
\end{proposition}
\begin{proof}
Let $s=\frac{q_1}{q}>1$. Then by applying H\"{o}lder's inequality with respect to conjugate pair of exponents $(s,s'=\frac{q_1}{q_1-q})$, we have
\begin{align*}
 \|f\|_{GN^{\theta_1}_{p,q}(M)}&=\sup\limits_{0<\varepsilon<1} \varepsilon^{\theta_1} \left(\int\limits_{0}^{1}\left(t^{\frac{1}{p}+\frac{\varepsilon}{2}}\bar{f}(t,M)\right)^{q}t^{\frac{q\varepsilon}{2}}\frac{dt}{t}\right)^{\frac{1}{q}} \\
&\le\sup\limits_{0<\varepsilon<1} \varepsilon^{\theta_1} 
\left(\int\limits_{0}^{1}\left(t^{\frac{1}{p}+\frac{\varepsilon}{2}}\bar{f}(t,M)\right)^{q_1}\frac{dt}{t}\right)^{\frac{1}{q_1}}\left(\int\limits_0^{1}t^{\frac{qs'\varepsilon}{2}-1}dt\right)^{\frac{1}{qs'}}\\
&\le c \sup\limits_{0<\varepsilon<1} \varepsilon^{\theta_1-\frac{1}{qs'}} 
\left(\int\limits_{0}^{1}\left(t^{\frac{1}{p}+\frac{\varepsilon}{2}}\bar{f}(t,M)\right)^{q_1}\frac{dt}{t}\right)^{\frac{1}{q_1}}\\
\end{align*}
\begin{align*}
&\le c \sup\limits_{0<\varepsilon<1} \varepsilon^{\theta_1-\frac{1}{q}+\frac{1}{q_1}} 
\left(\int\limits_{0}^{1}\left(t^{\frac{1}{p}+{\varepsilon}{}}\bar{f}(t,M)\right)^{q_1}\frac{dt}{t}\right)^{\frac{1}{q_1}}\\
&= c \sup\limits_{0<\varepsilon<1}\varepsilon^{\theta} \left(\int\limits_{0}^{1}\left(t^{\frac{1}{p}+{\varepsilon}}\bar{f} (t,M)\right)^{q_1}\frac{dt}{t}\right)^{\frac{1}{q_1}}\asymp\|f\|_{GN^{\theta}_{p,q_1}(M)}. 
\end{align*}
\end{proof}

The following theorem is H\"{o}lder type inequality for grand net spaces.
\begin{theorem}
    Let $1\le p_1,p_2<\infty$ and $1\le  s_1,s_2\le \infty$, $\theta\ge0$ and ${1}{}=\frac{1}{p_1}+\frac{1}{p_2}$, ${1}{}=\frac{1}{s_1}+\frac{1}{s_2}$. Then
\begin{equation}\label{L4}
\int_{0}^{1}\bar{f}(t,M)\bar{g}(t,M)dt \le \left\|f\right\|_{GN^{\theta}_{p_1,s_1}(M)}  \left\|g\right\|_{GN^{-\theta}_{p_2,s_2}(M)}.
\end{equation}
\end{theorem}
\begin{proof}
 The proof of \eqref{L4} follows from the application of H\"{o}lder's inequality.
\begin{align*}
    \int_{0}^{1}\bar{f}(t,M)\bar{g}(t,M)dt&=\int_{0}^{1}\left(t^{\frac{1}{p_1}+\varepsilon}\bar{f}(t,M)t^{\frac{1}{p_2}-\varepsilon}\bar{g}(t,M)\right)\frac{dt}{t}\\
    &\le \varepsilon^{\theta}\left(\int\limits_{0}^{1}\left(t^{\frac{1}{p_1}+{\varepsilon}}\bar{f}(t,M)\right)^{s_1}\frac{dt}{t}\right)^{\frac{1}{s_1}} \varepsilon^{-\theta}\left(\int\limits_{0}^{1}\left(t^{\frac{1}{p_2}-{\varepsilon}}\bar{g}(t,M)\right)^{s_2}\frac{dt}{t}\right)^{\frac{1}{s_2}}\\
    &\le \|f\|_{GN^{\theta}_{p_1,s_1}(M)}\varepsilon^{-\theta}\left(\int\limits_{0}^{1}\left(t^{\frac{1}{p_2}-{\varepsilon}}\bar{g}(t,M)\right)^{s_2}\frac{dt}{t}\right)^{\frac{1}{s_2}}.
\end{align*}
Since $\varepsilon>0$ is arbitrary, it implies
\begin{equation}
\int_{0}^{1}\bar{f}(t,M)\bar{g}(t,M)dt \le \left\|f\right\|_{GN^{\theta}_{p_1,s_1}(M)}  \left\|g\right\|_{GN^{-\theta}_{p_2,s_2}(M)}.
\end{equation}
\end{proof}
\subsection{Equivalence of Norms}
In this subsection we show that the grand Lorentz spaces introduced in \cite{DHN} are contained in the scale of grand net spaces. Let $f^*$ be non--increasing rearrangement of $f$. By $f^{**}$, we denote the maximal function of $f^*$ given by $f^{**}(t)=\frac{1}{t}\int\limits_0^tf^{*}(s)ds$. The grand Lorentz spaces $GL^{\theta}_{p,q}(\Omega)$ were introduced by Nursultanov et al. in \cite{DHN}. Like in the classical case, we  introduce the grand Lorentz space $\mathcal{GL}^{\theta}_{p,q}(\Omega)$ using $f^{**}$ instead of $f^{*}$ in the form
\begin{equation}
   \left\|f\right\|_{\mathcal{GL}^{\theta}_{p,q}(\Omega)} = 
\begin{cases} 
\sup\limits_{0<\varepsilon\le1}\varepsilon^{\theta}\left(\int\limits_{0}^{1}\left(t^{\frac{1}{p}+\varepsilon}{f^{**}}(t)\right)^{q}\frac{dt}{t}\right)^{\frac{1}{q}} & \text{if } \theta\ge0\;,\;0<p<\infty  \\ 
\inf\limits_{0<\varepsilon\le\frac{1}{p}}\varepsilon^{\theta}\left(\int\limits_{0}^{1}\left(t^{\frac{1}{p}-\varepsilon}{f^{**}}(t)\right)^{q}\frac{dt}{t}\right)^{\frac{1}{q}} & \text{if } \theta<0\;,\;0<p<\infty  \\ 
\sup\limits_{0<\varepsilon\le1}\varepsilon^{\theta}\left(\int\limits_{0}^{1}\left(t^{\varepsilon}{f^{**}}(t)\right)^{q}\frac{dt}{t}\right)^{\frac{1}{q}} & \text{if } \theta\ge0\;, p=\infty\;  \\ 
\end{cases} 
\end{equation}
if $q<\infty$ and
\begin{equation}
   \left\|f\right\|_{\mathcal{GL}^{\theta}_{p,q}(\Omega)} = 
\begin{cases} 
\sup\limits_{0<\varepsilon\le1}\sup\limits_{0<t<1}\varepsilon^{\theta}t^{\frac{1}{p}+\varepsilon}{f^{**}}(t,M) & \text{if } \theta\le0\;,\;0<p<\infty  \\ 
\inf\limits_{0<\varepsilon\le\frac{1}{p}}\sup\limits_{0<t<1}\varepsilon^{\theta}t^{\frac{1}{p}-\varepsilon}{f^{**}}(t,M) & \text{if } \theta<0\;,\;0<p<\infty   \\ 
\sup\limits_{0<\varepsilon\le1}\sup\limits_{0<t<1}\varepsilon^{\theta}t^{\varepsilon}{f^{**}}(t,M) & \text{if } \theta\ge0\;,\;0<p<\infty   \\ 
\end{cases} 
\end{equation}
if $q=\infty$.
\begin{remark}
    Similar to Proposition \ref{P1} (iv)-(v), when $\theta>0$ we have:
    
\begin{equation}\label {D1}
\left\|f\right\|_{\mathcal{GL}^{\theta}_{p,q}(\Omega)} \asymp\sup\limits_{0<\varepsilon\le\delta}\varepsilon^{\theta}\left(\int\limits_{0}^{1}\left(t^{\frac{1}{p}+\varepsilon}{f^{**}}(t)\right)^{q}\frac{dt}{t}\right)^{\frac{1}{q}}
\end{equation}
for $0<\delta<1$, and 
\begin{equation}\label {D2}
\left\|f\right\|_{\mathcal{GL}^{-\theta}_{p,q}(\Omega)} \asymp\inf\limits_{0<\varepsilon\le\delta}\varepsilon^{-\theta}\left(\int\limits_{0}^{1}\left(t^{\frac{1}{p}-\varepsilon}{f^{**}}(t)\right)^{q}\frac{dt}{t}\right)^{\frac{1}{q}}
\end{equation}
for $0<\delta<\frac{1}{p}$. 
\end{remark}
\begin{remark}\label{r5}
For any $\theta\in\mathbb{R}$, $0<p,q\le\infty$. The  equivalence of the norms 
\begin{equation*}
    \left\|f\right\|_{{GL}^{\theta}_{p,q}(\Omega)}\asymp    \left\|f\right\|_{\mathcal{GL}^{\theta}_{p,q}(\Omega)}
\end{equation*}
holds. 
The proof follows in a similar manner to that for the classical Lorentz spaces. Using the fact $f^*<f^{**}$, we have $  \left\|f\right\|_{{GL}^{\theta}_{p,q}(\Omega)}\le    \left\|f\right\|_{\mathcal{GL}^{\theta}_{p,q}(\Omega)}$. For the reverse inequality choosing $\delta=\frac{1}{2p'}$ in \eqref{D1} and applying Hardy's inequality we obtain
\begin{equation*}
\left\|f\right\|_{\mathcal{GL}^{\theta}_{p,q}(\Omega)}\le\sup\limits_{0<\varepsilon<\frac{1}{2p'}}\frac{1}{\left(\frac{1}{p'}-\varepsilon\right)}\left\|f\right\|_{{GL}^{\theta}_{p,q}(\Omega)}=c_{p'}\left\|f\right\|_{{GL}^{\theta}_{p,q}(\Omega)}.
\end{equation*}
\end{remark}
The following statement shows that the grand Lorentz space is a special case of grand net spaces. 
\begin{theorem}
Assume that $1<p,q<\infty$, $\theta\in\mathbb{R},$ and  $M^*$ is the family of all finite measure subsets of a domain $\Omega\subset\mathbb{R}^n$. Then
\end{theorem}
\begin{equation}\label{e1}
GN^{\theta}_{p,q}(M^*)={G{L}}^{\theta}_{p,q}(\Omega).
\end{equation}
\begin{proof}
By using the following inequality proved in \cite{nursultanov2011net} (and \cite{nursultanov1998net}) for the net $M^*$, 
\begin{equation*}
\bar{f}(t,M^*)\le f^{**}(t)\le 4\bar{f}(t/3,M^*),
\end{equation*}
we have $GN^{\theta}_{p,q}(M^*)=\mathcal{GL}^{\theta}_{p,q}(\Omega)$. Now, the result \eqref{e1} follows from Remark \ref{r5}.
\end{proof}
\section{Real Interpolation results}
In this section, we discuss the real interpolation results of grand net space, see \cite{bennett1988interpolation} for details. 

Let $(X_0,X_1)$ be a compatible pair of Banach spaces and the associated $K-$ functional is defined as:
\begin{equation*}
K(t,f;X_0,X_1)=\inf\limits_{f=f_0+f_1}\left(\|f_0\|_{X_0}+t\|f_1\|_{X_1}\right).
\end{equation*}
For $1\le q<\infty$ and $0<\eta<1$ we have
\begin{equation*}
(X_0,X_1)_{\eta,q}=\left\{f\in X_0+X_1: \|f\|_{(X_0,X_1)_{\eta,q}}=\left(\int\limits_{0}^{\infty}(t^{-\eta}K(t,f))^{q}\frac{dt}{t}\right)^{1/q}<\infty\right\}.
\end{equation*}
\begin{theorem}
Let $\theta>0$, $0<p_0<p_1<\infty$ and $0<q<\infty$. Then 
\begin{equation*}
\left(GN^{\theta}_{p_0,q_0}(M),GN^{\theta}_{p_1,q_1}(M)\right)_{\eta,q}\hookrightarrow GN^{\theta}_{p,q}(M),   
\end{equation*}
where $\frac{1}{p}=\frac{1-\eta}{p_0}+\frac{\eta}{p_1}$ and $0<\eta<1$.
\end{theorem}
\begin{proof}
Let $f\in GN^{\theta}_{p,q}(M)$ and $f=f_0+f_1$ an arbitrary representation of $f$, where $f_i\in GN^{\theta}_{p_i,q_i}(M)$ for $i=0,1$. Clearly, 
\begin{equation*}
  \bar{f}(t,M)\le\bar{f_0}(t,M)+\bar{f_1}(t,M).
\end{equation*}
Taking into account that $\varepsilon<\frac{1}{p_0} ,\frac{1}{p_1}$ and $N_{p_i,q_i}\hookrightarrow N_{p_i,\infty}$ we obtain
\begin{align*}
\bar{f}(t,M)&\le\bar{f_0}(t,M)+\bar{f_1}(t,M)\\
&\le t^{-\frac{1}{p_0}+\varepsilon}\|f_0\|_{N_{p_0(\varepsilon),\infty}(M)}+t^{-\frac{1}{p_1}+\varepsilon}\|f_1\|_{N_{p_1(\varepsilon),\infty}(M)}\\
&\le t^{-\frac{1}{p_0(\varepsilon)}}\|f_0\|_{N_{p_0(\varepsilon),\infty}(M)}+t^{-\frac{1}{p_1(\varepsilon)}}\|f_1\|_{N_{p_1(\varepsilon),\infty}(M)}\\
&\le t^{-\frac{1}{p_0(\varepsilon)}}\|f_0\|_{N_{p_0(\varepsilon),q_0}(M)}+t^{-\frac{1}{p_1(\varepsilon)}}\|f_1\|_{N_{p_1(\varepsilon),q_1}(M)},
\end{align*}
where $\frac{1}{p_i(\varepsilon)}=\frac{1}{p_i}+\varepsilon$. Since the representation $f=f_0+f_1$ is arbitrary, it yields
\begin{align*}
\bar{f}(t,M)&\le c t^{\frac{-1}{p_0(\varepsilon)}}\inf\limits_{f=f_0+f_1}\left(\|f_0\|_{N_{p_0(\varepsilon)},q_0}+t^{\frac{1}{p_0(\varepsilon)}-\frac{1}{p_{1}(\varepsilon)}}\|f_1\|_{N_{p_1(\varepsilon)},q_1}\right)\\
&= t^{\frac{-1}{p_0(\varepsilon)}}K(t^{\frac{1}{p_0(\varepsilon)}-\frac{1}{p_{1}(\varepsilon)}},f).
\end{align*}
Hence, we have
\begin{align*}
 \|f\|_{GN^{\theta}_{p,q}(M)}&=\sup\limits_{\varepsilon>0}{\varepsilon}^{\theta}\|f\|_{N_{p(\varepsilon),q}(M)}\\
  &=\sup\limits_{\varepsilon>0}{\varepsilon}^{\theta}\left(\int_{0}^{1}(t^{\frac{1}{p(\varepsilon)}}\bar{f}(t,M))^{q}\frac{dt}{t}\right)^{1/q}\\
  &\le \sup\limits_{\varepsilon>0}{\varepsilon}^{\theta}\left(\int_{0}^{1}(t^{\frac{1}{p(\varepsilon)}-\frac{1}{p_0(\varepsilon)}}K(t^{\frac{1}{p_0(\varepsilon)}-\frac{1}{p_{1}(\varepsilon)}},f))^{q}\frac{dt}{t}\right)^{1/q}\\
   &\le \sup\limits_{\varepsilon>0}{\varepsilon}^{\theta}\left(\int_{0}^{1}(t^{\frac{1}{p}-\frac{1}{p_0}}K(t^{\frac{1}{p_0}-\frac{1}{p_{1}}},f))^{q}\frac{dt}{t}\right)^{1/q}\\
     &\le \left(\frac{1}{p_0}-\frac{1}{p_1}\right)^{-1/q} \sup\limits_{\varepsilon>0}\left({\varepsilon}^{\theta}\int_{0}^{1}(t^{-\eta} K(t,f))^{q}\frac{dt}{t}\right)^{1/q}\\
      &\le \left(\frac{1}{p_0}-\frac{1}{p_1}\right)^{-1/q} \left(\int_{0}^{1}\left(t^{-\eta } \sup\limits_{\varepsilon>0}{\varepsilon}^{\theta}\inf\limits_{f=f_0+f_1}\left(\|f_0\|_{{N_{p_0(\varepsilon)},q_0}(M)}+\right.\right.\right.\\
      &\indent\left.\left.\left.+t\|f_1\|_{N_{p_1(\varepsilon)},q_1(M)}\right)\right)^{q}\frac{dt}{t}\right)^{1/q}\\ 
       &\le \left(\frac{1}{p_0}-\frac{1}{p_1}\right)^{-1/q} \left(\int_{0}^{1}\left(t^{-\eta } \inf\limits_{f=f_0+f_1}\left(\|f_0\|_{{GN}_{{p_0},q_0(M)}^{\theta}}+t\|f_1\|_{{GN}_{{p_1},q_1(M)}^{\theta}}\right)\right)^{q}\frac{dt}{t}\right)^{1/q}\\ 
       &\le c \|f\|_{({GN}^{\theta}_{p_0,q_0}(M), {GN}^{\theta}_{p_1,q_1}(M))_{\eta,q}}.
\end{align*}
\end{proof}
The following corollary is an immediate consequence of the latter theorem.
\begin{corollary}
Let $X_0$ and $X_1$ be a compatible pair and $1<p_0<p_1<\infty$. Let $T$ be quasi-linear operator such that 
\begin{equation*}
    T:X_0\to GN_{p_0,\infty}^{\theta}(M)\text{ with the norm } A_0
\end{equation*}
\begin{equation*}
    T:X_1\to GN_{p_1,\infty}^{\theta}(M)\text{ with the norm } A_1.
\end{equation*}
Then $T:(X_0,X_1)_{\eta,q}\to GN^{\theta}_{p,q}(M)$ with the norm $\|T\|\le c A_0^{1-\theta}A_1^{\theta}$.
\end{corollary}
\section{Integral Operators on Grand Net spaces}
In this section, we establish necessary and sufficient conditions for the boundedness of an integral operator acting from a Banach space of functions into the grand net spaces. Similar criteria for net spaces were previously proved in \cite[Section 4]{nursultanov2011net}.

Let $X(\Omega)$ be any Banach space of measurable functions defined on $\Omega$ endowed with a Lebesgue measure such that the space of compactly supported functions is dense in $X(\Omega)$. For brevity, we shall use $X$ instead of $X(\Omega)$. The associate space of $X$ is denoted by $X^{*}$ and is defined as:
\begin{equation}
X^{*}=\left\{h:\|h\|_{X^*}=\sup_{\|f\|_X\le1}\left|\int\limits_{\Omega}h(x)f(x)dx\right|<\infty\right\}.
\end{equation}
\begin{theorem}\label{t3}
Let $\theta\ge0$ and $1< q\le\infty$. Let $M$ be any net as defined in Section 1.  Let the integral operator 
\begin{equation}\label{e2}
Tf(y)=\int\limits_{\Omega}K(x,y)f(x)dx
\end{equation}
act from $X$ into $GN_{q,\infty}^{\theta}(M)$.
Then the necessary and sufficient condition for $T$ to be bounded from $X$ into $GN_{q,\infty}^{\theta}(M)$ is
\begin{equation}
\sup\limits_{0<t<1}\sup_{\omega\in M, |\omega|>t}\frac{t^{1/q}}{|ln(t)|^{\theta}}\left\|\frac{1}{|\omega|}\int\limits_{\omega}K(\cdot,y)dy\right\|_{X^{*}}<\infty.
\end{equation}
\end{theorem}
\begin{proof}
    From the definition of $X^*$ and Theorem \ref{t1}(a), we have the following estimates:
    \begin{align*}
    \|T\|_{X\to GN_{q,\infty}^{\theta}(M)}&=\sup_{\|f\|_X\le1}\|Tf\|_{GN_{q,\infty}^{\theta}(M)}\\
    &\asymp \sup_{\|f\|_X\le1}\sup_{t>0}\frac{t^{\frac{1}{q}}}{|\ln{t}|^{\theta}}\overline{Tf}(t,M),\\
    &=\sup_{\|f\|_X\le1}\sup_{t>0}\frac{t^{\frac{1}{q}}}{|\ln{t}|^{\theta}} \sup_{\omega\in M, |\omega|>t}\frac{1}{|\omega|}\left|\int\limits_{\omega}Tf(y)dy\right|\\
    &=\sup_{\|f\|_X\le1}\sup_{t>0}\frac{t^{\frac{1}{q}}}{|\ln{t}|^{\theta}} \sup_{\omega\in M, |\omega|>t}\frac{1}{|\omega|}\left|\int\limits_{\omega}\int\limits_{\mathbb{R}^n}K(x,y)f(x)dxdy\right|\\
     &=\sup_{t>0}\sup_{\omega\in M, |\omega|>t} \frac{t^{\frac{1}{q}}}{|\ln{t}|^{\theta}} \sup_{\|f\|_X\le1}\left|\int\limits_{\mathbb{R}^n}f(x)\frac{1}{|\omega|}\int\limits_{\omega}K(x,y)dydx\right|\\
      &=\sup_{t>0}\sup_{\omega\in M, |\omega|>t} \frac{t^{\frac{1}{q}}}{|\ln{t}|^{\theta}} \left\|\frac{1}{|\omega|}\int\limits_{\omega}K(\cdot,y)dy\right\|_{X^{*}}    
    \end{align*}
\end{proof}
\begin{remark}
The result for case $\theta=0$, that is, net spaces, was proved in \cite{nursultanov2011net}.
\end{remark}
The following corollary presents the criteria for the integral operator to be $(p,q)$-quasi weak type operator:
\begin{corollary}
    Let $1<p,q\le\infty$ and $\theta_1,\theta_2\ge0$. Let $M_1=\{\omega\subset\Omega_1: 0<|\omega|<\infty\}$, $M_2=\{e\subset\Omega_2: 0<|e|<\infty\}$. Then the sufficient condition for the operator \eqref{e2} to be $(p,q)$-quasi weak type operator in Grand Lorentz spaces i.e. $T:{GL^{\theta_1}_{p,1}(\Omega_1)\to GL^{\theta_2}_{q,\infty}(\Omega_2)}$  is bounded if,
    \begin{equation}
    \sup\limits_{\substack{w\in M_1\\ t_1>0\\|w|>t_1}}\sup\limits_{\substack{e\in M_2\\ t_2>0\\|e|>t_2}} \frac{t_2^{1/q}}{|\ln t_2|^{\theta_2}}{t_1^{1/p'}}{|\ln t_1|^{\theta_1}}\left|\frac{1}{|w|}\frac{1}{|e|}\int\limits_{w}\int\limits_{e}K(x,y)dydx\right|<\infty.    
    \end{equation}
\end{corollary}
\begin{proof}
By using Theorem \ref{t3} and the fact that $(GL_{p,1}^{\theta})^{*}=GL_{p',\infty}^{-\theta}$ we have
\begin{align*}
\|T\|_{GL^{\theta_1}_{p,1}(\Omega_1)\to GL^{\theta_2}_{q,\infty}(\Omega_2)}&\asymp \|T\|_{GL^{\theta_1}_{p,1}(\Omega_1)\to GN^{\theta_2}_{q,\infty}(M_2)}\\
&=\sup\limits_{\substack{e\in M_2\\ t_2>0\\|e|>t_2}}\frac{t_2^{1/q}}{|\ln t_2|^{\theta_2}}\left\|\frac{1}{|e|}\int\limits_{e}K(\cdot,y)dy\right\|_{GL^{-\theta_1}_{p',\infty}(\Omega_1)}\\
&=\sup\limits_{\substack{e\in M_2\\ t_2>0\\|e|>t_2}}\frac{t_2^{1/q}}{|\ln t_2|^{\theta_2}}\left\|\frac{1}{|e|}\int\limits_{e}K(\cdot,y)dy\right\|_{GL^{-\theta_1}_{p',\infty}(\Omega_1)}\\
&=\sup\limits_{\substack{w\in M_1\\ t_1>0\\|w|>t_1}}\sup\limits_{\substack{e\in M_2\\ t_2>0\\|e|>t_2}}\frac{t_2^{1/q}}{|\ln t_2|^{\theta_2}}{t_1^{1/p'}}{|\ln t_1|^{\theta_1}}\times\\
&\indent\times\left|\frac{1}{|w|}\frac{1}{|e|}\int\limits_{w}\int\limits_{e}K(x,y)dydx\right|.
\end{align*}
\end{proof}

\section{Acknowledgement}

The authors express their gratitude to Professor Erlan Nursultanov for his valuable insights during the initial stages of this work. This research was funded by Nazarbayev University under Collaborative Research Program Grant 20122022CRP1601.

\end{document}